\documentclass[11pt]{article}

\usepackage[T1]{fontenc}
\usepackage[utf8]{inputenc}
\usepackage{lmodern}
\usepackage{amsmath,amssymb,amsthm,mathtools}
\usepackage{enumitem}
\usepackage{microtype}
\usepackage[colorlinks=true,linkcolor=blue,citecolor=blue,urlcolor=blue]{hyperref}
\usepackage[margin=1in]{geometry}
\hypersetup{pdftitle={Strong Subgraph-Count Stability in C2l+1-Free Graphs},pdfauthor={Yuanpei Wang and Xiaomiao Zhao},pdfkeywords={Generalized Turan problem, Stability, Odd cycle, Non-bipartite}}
\setlist{leftmargin=2em}

\newtheorem{theorem}{Theorem}[section]
\newtheorem*{knownresult}{Theorem}
\newtheorem{lemma}[theorem]{Lemma}
\newtheorem*{lemma*}{Lemma}
\newtheorem{proposition}[theorem]{Proposition}
\newtheorem{corollary}[theorem]{Corollary}
\theoremstyle{definition}
\newtheorem{definition}[theorem]{Definition}
\theoremstyle{plain}

\theoremstyle{remark}

\theoremstyle{plain}
\newtheorem{problem}[theorem]{Problem}
\numberwithin{equation}{section}

\DeclareMathOperator{\ex}{ex}
\DeclareMathOperator{\Aut}{Aut}
\newcommand{\cG}{\mathcal G}
\newcommand{\cB}{\mathcal B}
\newcommand{\cS}{\mathcal S}
\newcommand{\cT}{\mathcal T}
\newcommand{\cR}{\mathcal R}
\newcommand{\N}{\mathcal N}

\title{Strong Subgraph-Count Stability in $C_{2\ell+1}$-Free Graphs}
\author{Yuanpei Wang$^{1,*}$\quad Xiaomiao Zhao$^{2,\dagger}$\\[0.5ex]
\small $^1$ Department of Mathematics, Shanghai University, Shanghai 200444, P.R. China\\
\small $^2$ Department of Mathematical Sciences, Tsinghua University, Beijing 100084, P.R. China}
\date{}

\begin{document}
\maketitle
\begingroup
\renewcommand{\thefootnote}{\fnsymbol{footnote}}
\footnotetext[1]{boyuan@shu.edu.cn}
\footnotetext[2]{\emph{Corresponding author}: zxm23@mails.tsinghua.edu.cn}
\endgroup

\begin{abstract}
Starting from the stability theorem of Erd\H{o}s and Simonovits, stability problems for graphs forbidding a fixed subgraph have been studied in terms of edge numbers, spectral radii and subgraph counts. Let $\N(F,G)$ denote the number of unlabeled copies of $F$ in $G$. It is known that, for every fixed path $P_t$ and even cycle $C_{2a}$, the maximum number of copies in an $n$-vertex $C_{2\ell+1}$-free graph is attained by the bipartite Tur\'an graph $T_{n,2}$.

In this paper we obtain strong structural stability for $C_{2\ell+1}$-free graphs in terms of copies of paths and even cycles. For fixed $\ell\ge2$ and $3\le r\le2\ell-1$, we show that if an $n$-vertex $C_{2\ell+1}$-free graph contains at least as many copies of $P_t$ or $C_{2a}$ as the corresponding suspended extremal construction, then it has the corresponding suspension structure. This gives exact high-chromatic extremal theorems for paths and even cycles.

We also prove a counting theorem for nearly complete bipartite graphs. It shows that, for every fixed matching-admissible connected bipartite graph $F$, both imbalance between the two parts and missing cross-edges decrease the number of copies of $F$ by a term with a specified main coefficient. This theorem is independent of the forbidden odd cycle and converts subgraph-count assumptions into the edge bounds needed for the structural theorem.
\end{abstract}


\section{Introduction}

For graphs $G$ and $F$, say that $G$ is \emph{$F$-free} if $G$ contains no subgraph isomorphic to $F$. For graphs $H$ and $G$, write $\N(H,G)$ for the number of unlabeled copies of $H$ in $G$. Define
\[
        \ex(n,H,F)=\max\{\N(H,G): G \text{ is an $n$-vertex $F$-free graph}\}.
\]

When $H=K_2$, we have $\N(K_2,G)=e(G)$, and therefore $\ex(n,K_2,F)$ is the \emph{ordinary Tur\'an number}. Mantel's theorem \cite{Mantel1907} gives $\ex(n,K_2,K_3)=\lfloor n^2/4\rfloor$, with extremal graph $T_{n,2}$. Tur\'an's theorem \cite{Turan1941} extends this to cliques: the unique $n$-vertex $K_{r+1}$-free graph with the maximum number of edges is the complete balanced $r$-partite graph $T_{n,r}$. More generally, the Erd\H{o}s--Stone--Simonovits theorem \cite{ErdosStone1946,ErdosSimonovits1966} determines the asymptotic value of $\ex(n,K_2,F)$ for every non-bipartite graph $F$ in terms of $\chi(F)$. The Erd\H{o}s--Simonovits stability theorem \cite{Simonovits1968} gives the corresponding structural statement for graphs with nearly extremal number of edges.

The case $H\ne K_2$ asks for subgraph counts rather than only edge counts. This \emph{generalized Tur\'an problem} goes back at least to Erd\H{o}s's theorem on the number of cliques in clique-free graphs \cite{Erdos1962}, and was systematically developed by Alon and Shikhelman \cite{AlonShikhelman2016}. For a recent survey, see Gerbner and Palmer \cite{GerbnerPalmerSurvey2026}.

A central question is when the corresponding Tur\'an graph remains extremal for $\ex(n,H,F)$. The following theorem gives such a result for the paths and even cycles considered in this paper.

Here and throughout the paper, $P_t$ denotes the path with $t$ vertices, and $C_t$ denotes the cycle with $t$ vertices.

\begin{knownresult}[Gerbner \cite{Gerbner2021}; Hei and Hou \cite{HeiHou2024}]
Fix integers $\ell\ge2$, $t\ge2$ and $a\ge2$. Then
\[
        \ex(n,P_t,C_{2\ell+1})=\N(P_t,T_{n,2}),
        \qquad
        \ex(n,C_{2a},C_{2\ell+1})=\N(C_{2a},T_{n,2})
\]
for all sufficiently large $n$. For $a=2$, equality in the second formula holds only for $G=T_{n,2}$.
\end{knownresult}

F\"uredi and Gunderson \cite{FurediGunderson2015} proved that, for $\ell\ge2$ and $n\ge4\ell-2$, every $n$-vertex $C_{2\ell+1}$-free graph has at most $\lfloor n^2/4\rfloor$ edges, and for $n\ge4\ell$ the unique extremal graph is $T_{n,2}$. It is then natural to ask for the largest possible number of edges under the additional condition $\chi(G)\ge r$.

Ren, Wang, Wang and Yang \cite{RenWangWangYang2024} proved that, for $3\le r\le2\ell$ and all sufficiently large $n$, every $n$-vertex $C_{2\ell+1}$-free graph $G$ with $\chi(G)\ge r$ satisfies $e(G)\le \lfloor(n-r+1)^2/4\rfloor+\binom r2$, with equality only for the graph obtained from $T_{n-r+1,2}$ by suspending a clique $K_r$. Yan and Peng \cite{YanPengStructural2024} strengthened the corresponding stability theorem to a suspension description, and Zou, Li and Peng \cite{ZouLiPeng2025} further refined this structure theorem and established a strong spectral analogue. Yan and Peng \cite{YanPengC42026} proved a counting analogue for $C_4$: sufficiently many copies of $C_4$ imply the same type of structural conclusion for graphs forbidding an odd cycle. In the special case $a=2$, the high-chromatic part of our even-cycle result recovers the suspended-clique extremal family from this $C_4$ theory; the present paper considers all fixed even cycles and also separates the unrestricted case $r<2a$.

The generalized Tur\'an theorem for paths and even cycles states that the bipartite Tur\'an graph is extremal. It does not address the corresponding non-bipartite or high-chromatic extremal problems. This paper proves such strong stability results for paths and even cycles.

The main counting tool is a theorem for nearly complete bipartite graphs. It shows that, for every fixed matching-admissible connected bipartite graph $F$, imbalance between the two parts and missing cross-edges decrease the number of copies of $F$ by a term specified in the theorem, up to an error term of the same order. We use this theorem only for $F=P_t$ and $F=C_{2 a}$. It converts a lower bound on subgraph counts into the edge threshold required by the strong stability theorem of Zou, Li and Peng~\cite{ZouLiPeng2025}.

We now state the main results. We first introduce the notation appearing in the statements.

\begin{definition}\label{def:suspension-family}
Let $B$ and $H$ be graphs, and let $x\in V(B)$. \emph{Suspending $H$ on $B$ at $x$} means taking the disjoint union of $B$ and $H$, identifying $x$ with one vertex of $H$, and adding no further edges; the vertex $x$ is called the \emph{suspension vertex}. For integers $n$ and $r\ge2$, let $\cG_{n,r}$ be the family of all $n$-vertex graphs which can be obtained from a bipartite graph $B$ by suspending graphs $H_1,\ldots,H_s$ on $B$ so that the sets $V(H_i)\setminus V(B)$ are pairwise disjoint and
\[
        \sum_{i=1}^s |V(H_i)\setminus V(B)|\le r-2.
\]
The graph $B$ is called a \emph{bipartite core} of the resulting graph, and the vertices in $\bigcup_i(V(H_i)\setminus V(B))$ are called \emph{outside vertices}.
\end{definition}

Let $\cT^*(r,n)$ be the family of graphs obtained from the bipartite Tur\'an graph $T_{n-r+1,2}$ by suspending a clique $K_r$. For integers $t\ge4$ and $r\ge3$, let $\cS_{t,r}(n)$ be the subfamily of $\cT^*(r,n)$ consisting of the graphs with the maximum number of copies of $P_t$.

\begin{proposition}\label{prop:Tstar-path-choice}
The family $\cS_{t,r}(n)$ is as follows. If $n-r+1$ is even, then $\cS_{t,r}(n)=\cT^*(r,n)$. If $n-r+1$ is odd, let $T_+(r,n)$ and $T_-(r,n)$ be the graphs obtained by suspending $K_r$ at a vertex in the larger and smaller part of $T_{n-r+1,2}$, respectively. If $t$ is odd, then $\cS_{t,r}(n)=\{T_-(r,n)\}$. If $t=2q$ is even, then
\[
        \cS_{t,r}(n)=
        \begin{cases}
        \{T_+(r,n)\}, & r\le q+1,\ (t,r)\ne(4,3),\\
        \{T_+(r,n),T_-(r,n)\}, & (t,r)=(4,3),\\
        \{T_-(r,n)\}, & r\ge q+2.
        \end{cases}
\]
\end{proposition}
For cycles, no such choice is needed: all members of $\cT^*(r,n)$ have the same number of copies of $C_{2a}$, because the bipartite core and the suspended clique meet in a single cut vertex, and hence each cycle is contained either in the bipartite core or in the suspended clique.

For $r\ge3$, every graph in $\cG_{n,r}$ is $(r-1)$-colourable: colour a bipartite core with two colours, and colour the outside vertices of each suspended graph with distinct colours different from the colour of its suspension vertex; there are enough such colours since the total number of outside vertices is at most $r-2$.
Consequently, $\cT^*(r,n)\subseteq \cG_{n,r+1}\setminus\cG_{n,r}$, since each member of $\cT^*(r,n)$ is obtained with $r-1$ outside vertices and contains a clique $K_r$.

For integers $t\ge4$, $a\ge2$ and $r\ge3$, define
\[
\begin{aligned}
        b^{\mathrm P}_{t,r}(n)
        &=\max\{\N(P_t,H):H\in\cG_{n,r+1}\setminus\cG_{n,r}\},\\
        b^{\mathrm P,\chi}_{t,r}(n)
        &=\max\{\N(P_t,H):H\in\cG_{n,r+1}\setminus\cG_{n,r},\ \chi(H)\ge r\},\\
        b^{\circ}_{a,r}(n)
        &=\max\{\N(C_{2a},H):H\in\cG_{n,r+1}\setminus\cG_{n,r}\},\\
        b^{\circ,\chi}_{a,r}(n)
        &=\max\{\N(C_{2a},H):H\in\cG_{n,r+1}\setminus\cG_{n,r},\ \chi(H)\ge r\}.
\end{aligned}
\]
Let $\cB^{\mathrm P}_{t,r}(n)$, $\cB^{\mathrm P,\chi}_{t,r}(n)$, $\cB^{\circ}_{a,r}(n)$ and $\cB^{\circ,\chi}_{a,r}(n)$ denote the corresponding sets of extremal graphs, respectively.

We now state the path theorem. The extra condition $r\le t-1$ is needed only for odd $t$ without the chromatic condition.

\begin{theorem}\label{thm:path-main}
Fix integers $\ell\ge2$, $t\ge4$ and $3\le r\le2\ell-1$. If $t$ is odd, assume also that $r\le t-1$. There exists $n_0=n_0(\ell,t,r)$ such that the following holds for every $n\ge n_0$. If $G$ is an $n$-vertex $C_{2\ell+1}$-free graph and
\[
        \N(P_t,G)\ge b^{\mathrm P}_{t,r}(n),
\]
then either $G\in\cG_{n,r}$ or $G\in\cS_{t,r}(n)$.
\end{theorem}

\begin{corollary}\label{cor:path-chromatic}
Fix integers $\ell\ge2$, $t\ge4$ and $3\le r\le2\ell-1$. For all sufficiently large $n$,
\[
        \max_{\substack{|V(G)|=n,\ C_{2\ell+1}\nsubseteq G\\ \chi(G)\ge r}} \N(P_t,G)=\N(P_t,S)
\]
for every $S\in\cS_{t,r}(n)$. Moreover, the extremal graphs are precisely the members of $\cS_{t,r}(n)$.
\end{corollary}

The case of $C_4$ has already been studied by Yan and Peng \cite{YanPengC42026}, who proved that a sufficiently large number of copies of $C_4$ in a $C_{2\ell+1}$-free graph implies strong structural stability. The next theorem extends this counting stability framework from $C_4$ to every fixed even cycle $C_{2a}$.

For the even-cycle result with $r<2a$, let $\cR_r(n)$ be the set of graphs $R\in\cG_{n,r+1}\setminus\cG_{n,r}$ which have a bipartite core isomorphic to $T_{n-r+1,2}$. Equivalently, $R$ is obtained from $T_{n-r+1,2}$ by suspending graphs with $r-1$ outside vertices in total. Since $r<2a$, every copy of $C_{2a}$ in such a graph is contained in the bipartite core.

\begin{theorem}\label{thm:cycle-main}
Fix integers $\ell\ge2$, $a\ge2$ and $3\le r\le2\ell-1$. There exists $n_0=n_0(\ell,a,r)$ such that the following holds for every $n\ge n_0$. If $G$ is an $n$-vertex $C_{2\ell+1}$-free graph and
\[
        \N(C_{2a},G)\ge b^{\circ}_{a,r}(n),
\]
then either $G\in\cG_{n,r}$, or $G\in\cR_r(n)$ when $r<2a$, and $G\in\cT^*(r,n)$ when $r\ge2a$.
\end{theorem}

\begin{corollary}\label{cor:cycle-chromatic}
Fix integers $\ell\ge2$, $a\ge2$ and $3\le r\le2\ell-1$. For all sufficiently large $n$,
\[
        \max_{\substack{|V(G)|=n,\ C_{2\ell+1}\nsubseteq G\\ \chi(G)\ge r}} \N(C_{2a},G)=\N(C_{2a},T)
\]
for every $T\in \cT^*(r,n)$. Moreover, the extremal graphs are precisely the members of $\cT^*(r,n)$.
\end{corollary}

Write $t_{s,2}=e(T_{s,2})$. 
A connected bipartite graph $F$ is called \emph{matching-admissible} if its colour classes $A,B$ may be labelled so that
\begin{equation}
        |A|=k+\nu,\qquad |B|=k,\qquad \nu\in\{0,1\},
        \label{eq:admissible-sizes}
\end{equation}
for some $k\ge1$, and $F$ has a matching of size $k$ saturating $B$. Let $v(F)=2k+\nu$, and let $\Aut(F)$ denote the automorphism group of $F$. Define
\begin{equation}
        \kappa_F=
        \begin{cases}
        \dfrac{2k}{|\Aut(F)|\,2^{2k-2}}, & \nu=0,\\[1.1ex]
        \dfrac{k}{|\Aut(F)|\,2^{2k-2}}, & \nu=1.
        \end{cases}
        \label{eq:kappa-F}
\end{equation}

\begin{theorem}\label{thm:local-counting}
Let $F$ be a fixed matching-admissible connected bipartite graph, and let $C_0>0$. If $H$ is a bipartite graph with bipartition $X\cup Y$, where $|X|=x$, $|Y|=y$, $s=x+y$, and $t_{s,2}-e(H)\le C_0s$, then
\[
        \N(F,H)\le \N(F,T_{s,2})-
        \kappa_F\bigl(t_{s,2}-e(H)\bigr)s^{v(F)-2}
        +O_{F,C_0}(s^{v(F)-2}).
\]
\end{theorem}

All asymptotic notation is with respect to $n\to\infty$; the hidden constants may depend on fixed parameters. For $W\subseteq V(G)$, write $G[W]$ for the induced subgraph on $W$, and let $d_W(v)$ be the number of neighbours of $v$ in $W$.

The rest of the paper is organized as follows. Section~\ref{sec:prelim} records the extremal reductions and known stability results used later. Section~\ref{sec:stability} proves Theorem~\ref{thm:local-counting}. Section~\ref{sec:counting-bounds} gives the counting bounds used in the main proofs. Sections~\ref{sec:proofs-paths} and~\ref{sec:proofs-cycles} prove the path and even-cycle results, respectively. Section~\ref{sec:odd-remark} contains concluding remarks.

\section{Preliminaries}\label{sec:prelim}

We use the following theorem of Zou, Li and Peng \cite{ZouLiPeng2025}.

\begin{theorem}[Zou, Li and Peng \cite{ZouLiPeng2025}]\label{thm:zlp}
Let $\ell\ge2$, $2\le r\le2\ell-1$ and $n\ge100\ell$. If $G$ is an $n$-vertex $C_{2\ell+1}$-free graph and
\[
        e(G)\ge \left\lfloor\frac{(n-r)^2}{4}\right\rfloor+\binom{r+1}{2},
\]
then $G\in\cG_{n,r+1}$, unless $G$ is obtained from $T_{n-r,2}$ by suspending a copy of $K_{r+1}$.
\end{theorem}

We also use the following consequence of Lemma~3.3 in \cite{ZouLiPeng2025}.

\begin{lemma}[Zou, Li and Peng \cite{ZouLiPeng2025}]\label{lem:zlp-core}
Let $\ell\ge2$ and $c\ge2$ be fixed. For all sufficiently large $n$, every $n$-vertex $C_{2\ell+1}$-free graph $G$ with $e(G)\ge (n-c)^2/4$ has a set $W\subseteq V(G)$ such that $H=G[W]$ is bipartite, $|W|\ge n-2c$, $\delta(H)\ge11c$, and $d_W(v)<11c$ for every $v\in V(G)\setminus W$.
\end{lemma}

We also use the following spectral bound.

\begin{lemma}[Zhai, Lin and Shu \cite{ZhaiLinShu2021}]\label{lem:zls}
For every fixed $\ell\ge2$, every $C_{2\ell+1}$-free graph $G$ with $m$ edges satisfies
\[
        \lambda(G)\le \frac{\ell-\frac12+\sqrt{4m+(\ell-\frac12)^2}}{2}\le\sqrt m+\ell,
\]
where $\lambda(G)$ is the spectral radius of the adjacency matrix of $G$.
\end{lemma}

We next prove the proposition from the introduction which identifies the members of $\cT^*(r,n)$ with the largest number of copies of $P_t$.

\noindent\textbf{Proposition~\ref{prop:Tstar-path-choice}.}
Fix integers $t\ge4$ and $r\ge3$. If $n-r+1$ is even, then $\cS_{t,r}(n)=\cT^*(r,n)$. If $n-r+1$ is odd, let $T_+(r,n)$ and $T_-(r,n)$ be the graphs obtained by suspending $K_r$ at a vertex in the larger and smaller part of $T_{n-r+1,2}$, respectively. If $t$ is odd, then $\cS_{t,r}(n)=\{T_-(r,n)\}$. If $t=2q$ is even, then
\[
        \cS_{t,r}(n)=
        \begin{cases}
        \{T_+(r,n)\}, & r\le q+1,\ (t,r)\ne(4,3),\\
        \{T_+(r,n),T_-(r,n)\}, & (t,r)=(4,3),\\
        \{T_-(r,n)\}, & r\ge q+2.
        \end{cases}
\]

\begin{proof}
Fix $t\ge4$ and $r\ge3$, and let $d=r-1$. For $a,b\ge1$, let $A$ and $B$ be the parts of $K_{a,b}$, with $|A|=a$ and $|B|=b$. Recall that $\N(P_{2q},K_{a,b})=(a)_q(b)_q$ and $\N(P_{2q+1},K_{a,b})=((a+b-2q)/2)(a)_q(b)_q$.

For $L\ge2$, let $R_L^A(a,b)$ be the number of paths on $L$ vertices in $K_{a,b}$ with one fixed endpoint in $A$, and define $R_L^B(a,b)$ analogously. Then
\[
        R_{2p}^A(a,b)=(b)_p(a-1)_{p-1},
        \qquad
        R_{2p+1}^A(a,b)=(b)_p(a-1)_p,
\]
and $R_L^B(a,b)$ is obtained by reversing $a$ and $b$.

For $Z\in\{A,B\}$, let $T_Z(a,b)$ be the graph obtained from $K_{a,b}$ by suspending $K_r$ at a vertex in $Z$. A copy of $P_t$ in $T_Z(a,b)$ is either contained in $K_{a,b}$, contained in the suspended clique, or uses $j$ outside vertices followed by a path in $K_{a,b}$ whose endpoint is the suspension vertex. Therefore
\begin{equation}\label{eq:Tstar-path-count}
        \N(P_t,T_Z(a,b))
        =\N(P_t,K_{a,b})+\N(P_t,K_r)
        +\sum_{j=1}^{\min\{d,t-2\}}(d)_jR^Z_{t-j}(a,b).
\end{equation}
The first two terms in \eqref{eq:Tstar-path-count} do not depend on $Z$, and hence
\begin{equation}\label{eq:Tstar-path-difference}
        \N(P_t,T_A(a,b))-\N(P_t,T_B(a,b))
        =\sum_{j=1}^{\min\{d,t-2\}}(d)_j
        \bigl(R^A_{t-j}(a,b)-R^B_{t-j}(a,b)\bigr).
\end{equation}

If $a=b$, then $T_A(a,b)$ and $T_B(a,b)$ are isomorphic. It remains to consider $a=b+1$, with $A$ the larger part. Write
\[
        D_t(a,b)=\N(P_t,T_A(a,b))-\N(P_t,T_B(a,b)).
\]
From the formulas for $R_L^A$ and $R_L^B$,
\[
\begin{aligned}
        R_{2p}^A(a,b)-R_{2p}^B(a,b)
        &=-b^{2p-2}+O(b^{2p-3}),\\
        R_{2p+1}^A(a,b)-R_{2p+1}^B(a,b)
        &=p b^{2p-2}+O(b^{2p-3}).
\end{aligned}
\]
Using \eqref{eq:Tstar-path-difference}, we obtain
\[
        D_{2q+1}(a,b)=-d b^{2q-2}+O(b^{2q-3})
\]
and
\[
        D_{2q}(a,b)=d(q-d)b^{2q-4}+O(b^{2q-5}).
\]
Thus, for odd $t$, the graph suspended at the smaller part has more copies of $P_t$. For $t=2q$, the graph suspended at the larger part has more copies of $P_t$ when $d<q$, while the graph suspended at the smaller part has more copies of $P_t$ when $d>q$.

It remains only to consider $d=q$ in the even case. If $q=2$, then $t=4$ and $d=2$, and a direct calculation gives
\[
\begin{aligned}
        D_4(a,b)
        &=2\bigl(R_3^A(a,b)-R_3^B(a,b)\bigr)
          +2\bigl(R_2^A(a,b)-R_2^B(a,b)\bigr)\\
        &=2\bigl(b(a-1)-a(b-1)\bigr)+2(b-a)=0.
\end{aligned}
\]
If $q\ge3$, the next non-zero term is
\[
        D_{2q}(a,b)=q(q-1)(q-2)b^{2q-6}+O(b^{2q-7})>0.
\]
Since $d=r-1$, this gives exactly the description of $\cS_{t,r}(n)$ in the introduction.
\end{proof}

The next two propositions determine the extremal graphs for the numbers defined in the introduction.

\begin{proposition}\label{prop:path-extremal}
Fix integers $t\ge4$ and $r\ge3$. For all sufficiently large $n$,
\[
        \cB^{\mathrm P}_{t,r}(n)=\cB^{\mathrm P,\chi}_{t,r}(n)=\cS_{t,r}(n).
\]
Consequently, for every $S\in\cS_{t,r}(n)$, $b^{\mathrm P}_{t,r}(n)=b^{\mathrm P,\chi}_{t,r}(n)=\N(P_t,S)$.
\end{proposition}

\begin{proof}
Let $d=r-1$. The family $\cG_{n,r}$ is closed under deleting edges. Hence, if $G\in\cG_{n,r+1}\setminus\cG_{n,r}$ and $G^+$ is obtained from $G$ by adding edges with $G^+\in\cG_{n,r+1}$, then $G^+\notin\cG_{n,r}$. Also, in any expression of a graph in $\cG_{n,r+1}\setminus\cG_{n,r}$ by a bipartite core and suspended graphs, the number of outside vertices is $d$.

Let $G\in\cG_{n,r+1}\setminus\cG_{n,r}$ attain the maximum number of copies of $P_t$. Choose a bipartite core $B_0$ and suspended graphs for $G$ so that the outside set has size $d$, and let $A\cup B$ be a bipartition of $B_0$ with $|A|=a$ and $|B|=b$. Since $\cT^*(r,n)\subseteq\cG_{n,r+1}\setminus\cG_{n,r}$ and every member of $\cT^*(r,n)$ has $\Theta(n^t)$ copies of $P_t$, we must have $a,b=\Omega(n)$.

Let $\widehat G$ be obtained from $G$ by replacing $B_0$ with $K_{a,b}$ and replacing each suspended graph by the clique on its vertex set. Then $\widehat G\in\cG_{n,r+1}\setminus\cG_{n,r}$ and $G\subseteq\widehat G$. Since $a,b=\Omega(n)$ and $t$ is fixed, every edge in $E(\widehat G)\setminus E(G)$ is contained in a copy of $P_t$ in $\widehat G$. Indeed, an added cross-edge inside the core can be extended alternately in the two parts of $K_{a,b}$ until it has $t$ vertices. If an added edge joins the suspension vertex $x$ of some suspended graph to an outside vertex $u$, then one starts with $u x$ and continues from $x$ along an alternating path in the core; for $t=4$ this gives a path $u x y x'$ with $y$ in the opposite part of the core and $x'$ in the same part as $x$. If an added edge joins two outside vertices $u$ and $v$ in the same suspended graph, then one starts with $u v x$, where $x$ is the suspension vertex, and continues in the core; for $t=4$ this gives $u v x y$. Since the two core parts have linear size, the required distinct core vertices can always be chosen for all sufficiently large $n$. In each case the constructed copy contains the added edge and is therefore absent from $G$. Hence $G=\widehat G$ by the maximality of $G$.

We now maximize the number of copies of $P_t$ among the graphs obtained in this way. Such a graph is determined by $a+b=n-d$, by a partition $d=d_1+\cdots+d_s$ of the outside vertices into suspended cliques, and by the parts of the core containing the corresponding suspension vertices. For fixed data $\pi$, write $F_\pi(a,b)=\N(P_t,G_\pi(a,b))$ for the corresponding polynomial in $a$ and $b$.

Using the notation from the comparison above, let $d_i$ be the number of outside vertices in the $i$th suspended clique, and let $Z_i\in\{A,B\}$ be the side containing its suspension vertex. For $Z\in\{A,B\}$, set $Q_t^Z(a,b)=2R_{t-2}^Z(a,b)$. Since paths using at least three outside vertices, or two outside vertices from different suspended cliques, contribute only $O(n^{t-4})$ copies, we have
\[
\begin{aligned}
        F_\pi(a,b)
        &=\N(P_t,K_{a,b})
        +\sum_{Z_i=A} d_i R_{t-1}^A(a,b)
        +\sum_{Z_i=B} d_i R_{t-1}^B(a,b)\\
        &\quad
        +\sum_{Z_i=A}\binom{d_i}{2}Q_t^A(a,b)
        +\sum_{Z_i=B}\binom{d_i}{2}Q_t^B(a,b)
        +O(n^{t-4}).
\end{aligned}
\]
We first choose the sizes of the two parts of the complete bipartite core. Since $a+b=n-d$ is fixed, the number $\N(P_t,K_{a,b})$ is maximized when $a$ and $b$ are as equal as possible. More precisely, if $|a-b|\ge2$, then moving one vertex from the larger part to the smaller part increases $\N(P_t,K_{a,b})$ by $\Theta(n^{t-2})$. The terms involving outside vertices change by only $O(n^{t-3})$. Hence no extremal graph has $|a-b|\ge2$, and the core is $T_{n-d,2}=T_{n-r+1,2}$.

Now fix this balanced core. The side of the suspension vertex and the partition $d=d_1+\cdots+d_s$ must be compared together. Indeed, when the two parts of the core have sizes differing by one, the difference between $R_{t-1}^A(a,b)$ and $R_{t-1}^B(a,b)$ may have the same order as the terms involving two outside vertices.

For $Z\in\{A,B\}$ and $1\le j\le d$, define $M_Z(j)=jR_{t-1}^Z(a,b)+\binom j2 Q_t^Z(a,b)$.
Then the expansion above can be written as
\[
        F_\pi(a,b)=\N(P_t,K_{a,b})+\sum_{i=1}^s M_{Z_i}(d_i)+O(n^{t-4}).
\]
Let $\Lambda_Z=R_{t-1}^Z(a,b)+((d-1)/2)Q_t^Z(a,b)$, and choose $Z^*\in\{A,B\}$ so that $\Lambda_{Z^*}$ is maximal. Since $a,b=n/2+O(1)$, we have $Q_t^Z(a,b)=2R_{t-2}^Z(a,b)=c_{t,Z}n^{t-3}+O(n^{t-4})$ for some constant $c_{t,Z}>0$. Hence
\[
\begin{aligned}
        \sum_{i=1}^s M_{Z_i}(d_i)
        &=\sum_{i=1}^s d_i\left(R_{t-1}^{Z_i}(a,b)+\frac{d_i-1}{2}Q_t^{Z_i}(a,b)\right)\\
        &\le \sum_{i=1}^s d_i\left(R_{t-1}^{Z_i}(a,b)+\frac{d-1}{2}Q_t^{Z_i}(a,b)\right)\\
        &\le d\Lambda_{Z^*}=M_{Z^*}(d).
\end{aligned}
\]
If $s\ge2$, then each $d_i<d$, and the first inequality loses at least
\[
        \frac12\sum_{i=1}^s d_i(d-d_i)Q_t^{Z_i}(a,b)=\Omega(n^{t-3}),
\]
which dominates the error term $O(n^{t-4})$. Thus no extremal graph can have $s\ge2$. Hence all outside vertices lie in one suspended clique. The comparison for $\cT^*(r,n)$ above now gives precisely the members of $\cS_{t,r}(n)$. Each contains a clique $K_r$ and is $r$-colourable, so $\chi(S)=r$ for every $S\in\cS_{t,r}(n)$; hence the same graphs are extremal in the high-chromatic problem.
\end{proof}

\begin{proposition}\label{prop:cycle-extremal}
Fix integers $a\ge2$ and $r\ge3$. For all sufficiently large $n$, $b^{\circ}_{a,r}(n)=b^{\circ,\chi}_{a,r}(n)=\N(C_{2a},T)$ for every $T\in \cT^*(r,n)$, and $\cB^{\circ,\chi}_{a,r}(n)=\cT^*(r,n)$.
Moreover, the unrestricted even-cycle extremal family is
\[
        \cB^{\circ}_{a,r}(n)=
        \begin{cases}
        \cR_r(n), & r<2a,\\
        \cT^*(r,n), & r\ge2a.
        \end{cases}
\]
\end{proposition}

\begin{proof}
Let $d=r-1$. As in the proof of Proposition~\ref{prop:path-extremal}, $\cG_{n,r}$ is closed under deleting edges. Hence any expression of a graph in $\cG_{n,r+1}\setminus\cG_{n,r}$ by a bipartite core and suspended graphs has exactly $d$ outside vertices; otherwise the graph would already belong to $\cG_{n,r}$.

Let $G\in\cB^{\circ}_{a,r}(n)$. Choose a bipartite core $B$ and suspended graphs for $G$, with outside set of size $d$, and let $X\cup Y$ be a bipartition of $B$ with $|X|=x$ and $|Y|=y$. Since the core and each suspended graph meet in a cut vertex, every copy of $C_{2a}$ is contained either in the core or in one suspended graph. Since $\cT^*(r,n)\subseteq\cG_{n,r+1}\setminus\cG_{n,r}$ and every member of $\cT^*(r,n)$ has $\Theta(n^{2a})$ copies of $C_{2a}$, extremality gives $x,y=\Omega(n)$.

Let $\widehat G$ be obtained from $G$ by replacing $B$ with $K_{x,y}$ and leaving the suspended graphs unchanged. Then $\widehat G\in\cG_{n,r+1}\setminus\cG_{n,r}$ and $G\subseteq\widehat G$. If $B\ne K_{x,y}$, choose $uv\in E(K_{x,y})\setminus E(B)$. The number of copies of $C_{2a}$ in $K_{x,y}$ containing $uv$ is $\Theta(n^{2a-2})$, and all of them are absent from $G$. Hence $\N(C_{2a},\widehat G)>\N(C_{2a},G)$, a contradiction. Thus the core is $K_{x,y}$ with $x+y=n-d$.

Now $\N(C_{2a},K_{x,y})=(x)_a(y)_a/(2a)$, whereas the suspended graphs contribute only $O(1)$ copies. Therefore, for all sufficiently large $n$, $\N(C_{2a},K_{x,y})$ is maximized only when $|x-y|\le1$. Thus the core is $T_{n-r+1,2}$.

We first consider the case $r<2a$. In this case every suspended graph has at most $d+1=r<2a$ vertices, and hence contains no copy of $C_{2a}$. Consequently every graph in $\cR_r(n)$ has exactly $\N(C_{2a},T_{n-r+1,2})$ copies of $C_{2a}$. This is also the value of every graph in $\cT^*(r,n)$, because $K_r$ has fewer than $2a$ vertices. The preceding paragraphs show that every unrestricted extremal graph must have core $T_{n-r+1,2}$ and exactly $d$ outside vertices, hence lies in $\cR_r(n)$. Conversely, every graph in $\cR_r(n)$ has the same value as $\cT^*(r,n)$, and is therefore extremal. Thus $\cB^{\circ}_{a,r}(n)=\cR_r(n)$ for $r<2a$.

Assume next that $r\ge2a$. Write the outside sizes of the suspended graphs as $d_1,\ldots,d_m$, where $d_i\ge1$ and $\sum_i d_i=d$. Their total contribution is at most $\sum_{i=1}^m \N(C_{2a},K_{d_i+1})$. Since $\N(C_{2a},K_u)=(u)_{2a}/(4a)$, we have
\[
        \sum_i \N(C_{2a},K_{d_i+1})\le \N(C_{2a},K_{d+1}),
\]
as follows. Let $D_i$ be the outside vertex set of the $i$th suspended graph, let $x_i$ be its suspension vertex, let $D=\bigcup_iD_i$, and let $x$ be one new vertex. A copy of $C_{2a}$ in the clique on $\{x_i\}\cup D_i$ is mapped to a copy of $C_{2a}$ in the clique on $\{x\}\cup D$ by replacing $x_i$ with $x$ and keeping all outside vertices unchanged. These maps are injective, and their images are pairwise disjoint for different $i$. Hence the inequality above follows.

Moreover, if $m\ge2$ and $d+1=r\ge2a$, then the clique on $\{x\}\cup D$ contains a copy of $C_{2a}$ using outside vertices from two distinct sets $D_i$. This cycle is not obtained from the preceding maps, so the inequality is strict. Thus equality is possible only when all outside vertices lie in one suspended graph. Equality also requires that graph to be complete, because every edge of $K_r$ lies in a copy of $C_{2a}$. Hence, when $r\ge2a$, the unrestricted extremal graphs are precisely the graphs in $\cT^*(r,n)$.

It remains to identify the high-chromatic extremal graphs. Every graph $T\in\cT^*(r,n)$ satisfies $\chi(T)=r$ and has the unrestricted extremal value just computed. Conversely, if $G\in\cB^{\circ,\chi}_{a,r}(n)$, then $G\in\cB^{\circ}_{a,r}(n)$ and $\chi(G)\ge r$. The bipartite core is 2-colourable and different suspended graphs meet it only at cut vertices, so some suspended graph $J$ must satisfy $\chi(J)\ge r$. Since the total number of outside vertices is $r-1$, this suspended graph has at most $r$ vertices. Hence $|V(J)|=r$ and $J=K_r$. All outside vertices lie in this suspended clique, and therefore $G\in\cT^*(r,n)$. This proves $\cB^{\circ,\chi}_{a,r}(n)=\cT^*(r,n)$ and completes the proof.
\end{proof}

\begin{lemma}\label{lem:path-extremal-asymp}
For fixed $q\ge2$ and $r\ge3$,
\begin{align}
        b^{\mathrm P}_{2q,r}(n)&=\N(P_{2q},T_{n,2})-\frac{q(r-1)}{2^{2q-1}}n^{2q-1}+O(n^{2q-2}),\label{eq:path-even-extremal-asymp}\\
        b^{\mathrm P}_{2q+1,r}(n)&=\N(P_{2q+1},T_{n,2})-\frac{(2q+1)(r-1)}{2^{2q+1}}n^{2q}+O(n^{2q-1}).\notag
\end{align}
The same asymptotic formulas hold for $b^{\mathrm P,\chi}_{2q,r}(n)$ and $b^{\mathrm P,\chi}_{2q+1,r}(n)$, respectively.
\end{lemma}

\begin{proof}
We first record the complete bipartite path count needed here. If $K_{a,b}$ is complete bipartite, then, for every $t\ge2$,
\[
        \N(P_t,K_{a,b})
        =\frac{(a)_{\lceil t/2\rceil}(b)_{\lfloor t/2\rfloor}
        +(a)_{\lfloor t/2\rfloor}(b)_{\lceil t/2\rceil}}{2}.
\]
Indeed, an ordered alternating path starting in the part of size $a$ can be chosen in $(a)_{\lceil t/2\rceil}(b)_{\lfloor t/2\rfloor}$ ways, and the analogous number with $a$ and $b$ reversed counts those starting in the other part. Each unlabeled path is counted twice, once in each direction. Thus $\N(P_{2q},K_{a,b})=(a)_q(b)_q$, and
\[
        \N(P_{2q+1},K_{a,b})
        =\frac{(a)_{q+1}(b)_q+(b)_{q+1}(a)_q}{2}
        =\frac{a+b-2q}{2}(a)_q(b)_q.
\]
With the two parts of $T_{n,2}$ of sizes $n/2+O(1)$ and $n/2+O(1)$, these formulas give, for every fixed integer $d$,
\begin{align}
        \N(P_{2q},T_{n,2})-\N(P_{2q},T_{n-d,2})&=\frac{qd}{2^{2q-1}}n^{2q-1}+O(n^{2q-2}),\label{eq:even-T-diff}\\
        \N(P_{2q+1},T_{n,2})-\N(P_{2q+1},T_{n-d,2})&=\frac{(2q+1)d}{2^{2q+1}}n^{2q}+O(n^{2q-1}).\label{eq:odd-T-diff}
\end{align}
For the even case, one keeps the next term in the expansion of $(x)_q$; the odd case follows from $\N(P_{2q+1},T_{m,2})=m^{2q+1}/2^{2q+1}+O(m^{2q})$ with $m=n$ and $m=n-d$.

Every graph in $\cG_{n,r+1}\setminus\cG_{n,r}$ can be written with exactly $r-1$ vertices outside a bipartite core, so the core has $n-r+1$ vertices. For $P_{2q}$, the number of copies using at least one outside vertex is $O(n^{2q-2})$: after choosing one outside vertex, the copy must also contain one of the $O(1)$ suspension vertices through which it meets the core, leaving at most $2q-2$ vertices to choose freely from the core. The copies entirely in the core are maximized by completing the core to a balanced complete bipartite graph. Hence $b^{\mathrm P}_{2q,r}(n)=\N(P_{2q},T_{n-r+1,2})+O(n^{2q-2})$, and \eqref{eq:path-even-extremal-asymp} follows from \eqref{eq:even-T-diff} with $d=r-1$.

For $P_{2q+1}$ the same argument gives an outside contribution $O(n^{2q-1})$, and therefore $b^{\mathrm P}_{2q+1,r}(n)=\N(P_{2q+1},T_{n-r+1,2})+O(n^{2q-1})$. Now apply \eqref{eq:odd-T-diff} with $d=r-1$.

The lower bounds for the high-chromatic quantities come from the corresponding graphs $T\in\cT^*(r,n)$, which satisfy $\chi(T)=r$; the upper bounds follow from the formulas just proved without the chromatic restriction.
\end{proof}

\begin{lemma}\label{lem:cycle-extremal-asymp}
For fixed $a\ge2$ and $r\ge3$,
\[
        b^{\circ}_{a,r}(n)=b^{\circ,\chi}_{a,r}(n)
        =\N(C_{2a},T_{n,2})-\frac{r-1}{2^{2a}}n^{2a-1}+O(n^{2a-2}).
\]
Moreover, for every $T\in \cT^*(r,n)$,
\[
        b^{\circ,\chi}_{a,r}(n)=\N(C_{2a},T).
\]
\end{lemma}

\begin{proof}
We use the following cycle counts. If $K_{x,y}$ is complete bipartite, then $\N(C_{2a},K_{x,y})=(x)_a(y)_a/(2a)$. Indeed, an ordered alternating $2a$-cycle in $K_{x,y}$ can be selected in $(x)_a(y)_a$ ways after fixing which side contains the first vertex. Each unlabeled cycle is counted $2a$ times by cyclic shifts and twice by reversal. Hence, for every fixed integer $d$,
\[
        \N(C_{2a},T_{n,2})-\N(C_{2a},T_{n-d,2})
        =\frac{d}{2^{2a}}n^{2a-1}+O(n^{2a-2}).
\]
Also, writing $K_u$ for the complete graph on $u$ vertices, the standard count of cyclic orderings gives $\N(C_{2a},K_u)=(u)_{2a}/(4a)$.

By Proposition~\ref{prop:cycle-extremal}, the extremal core is $T_{n-r+1,2}$, and the suspended part contributes only $O(1)$ copies of $C_{2a}$. Thus $b^{\circ}_{a,r}(n)=\N(C_{2a},T_{n-r+1,2})+O(1)$, and the preceding difference formula with $d=r-1$ gives the asserted asymptotic formula. The equality with the high-chromatic extremal value and the final assertion are also part of Proposition~\ref{prop:cycle-extremal}.
\end{proof}

\section{\texorpdfstring{Proof of Theorem~\ref{thm:local-counting}}{Proof of the counting theorem}}\label{sec:stability}

We now prove Theorem~\ref{thm:local-counting}. Let $X\cup Y$ be the bipartition of $H$, with $|X|=x$ and $|Y|=y$, let $s=x+y$, and define
\[
        \eta=t_{s,2}-xy,\qquad \mu=xy-e(H).
\]
Thus $\eta$ measures the loss caused by the imbalance of the complete bipartite graph $K_{x,y}$, while $\mu$ measures the number of missing cross-edges of $H$ inside $K_{x,y}$. We have $\eta+\mu=t_{s,2}-e(H)$.

\begin{proof}[Proof of Theorem~\ref{thm:local-counting}]
Write the colour classes of $F$ as in \eqref{eq:admissible-sizes}, and fix a matching $M$ of size $k$ saturating $B$. Since $H\subseteq K_{x,y}$, we have $e(H)\le xy$. Hence
\[
        0\le \eta=t_{s,2}-xy\le t_{s,2}-e(H)\le C_0s,
        \qquad
        0\le \mu\le C_0s.
\]
As $x+y=s$ and $t_{s,2}=\lfloor s^2/4\rfloor$, we also have
\[
        t_{s,2}-xy=\frac{(x-y)^2}{4}+O(1).
\]
Thus $|x-y|=O_{C_0}(s^{1/2})$, and in particular
\begin{equation}\label{eq:xy-near-balanced}
        x,y=\frac{s}{2}+O_{C_0}(s^{1/2}).
\end{equation}

\medskip
\noindent\emph{\textbf{Step 1: the loss caused by imbalance.}}
Since $F$ is connected and bipartite, every embedding of $F$ into $K_{x,y}$ sends its two colour classes to the two parts of $K_{x,y}$, in one of the two possible orders. Therefore
\begin{equation}
        \N(F,K_{x,y})=
        \frac{(x)_{k+\nu}(y)_k+(x)_k(y)_{k+\nu}}{|\Aut(F)|}.
        \label{eq:complete-F-count}
\end{equation}
Let $z=xy$ and $z_0=t_{s,2}$. For fixed $s$, the right-hand side of \eqref{eq:complete-F-count} is a polynomial $P_s(z)$ in $z$. This uses the relation $x+y=s$: every symmetric polynomial in $x$ and $y$ with $s$ fixed can be written as a polynomial in $xy$. Concretely,
\[
        (x)_k(y)_k=\prod_{i=0}^{k-1}(x-i)(y-i)
        =\prod_{i=0}^{k-1}(z-is+i^2).
\]
If $\nu=0$, then
\[
        P_s(z)=\frac{2}{|\Aut(F)|}z^k+R_s(z).
\]
If $\nu=1$, then
\[
        (x)_{k+1}(y)_k+(x)_k(y)_{k+1}
        =(s-2k)(x)_k(y)_k,
\]
and hence
\[
        P_s(z)=\frac{s}{|\Aut(F)|}z^k+R_s(z).
\]
In the range $z=z_0+O(s)$, the expansion above gives
\begin{equation}\label{eq:Ps-derivative-bounds}
        R_s'(z)=O_F(s^{v(F)-3}),\qquad
        P_s''(z)=O_F(s^{v(F)-4}),
\end{equation}
with the convention that the second bound is zero when $P_s$ has degree at most one. This follows because every non-leading term in the product loses at least one power of $z$ and gains at most a fixed power of $s$, while $z=\Theta(s^2)$ in the range under consideration.

Since $z=z_0-\eta$ and $\eta=O_{C_0}(s)$, Taylor expansion at $z_0$ gives
\[
        P_s(z)=P_s(z_0)-\eta P_s'(z_0)+O_{F,C_0}(s^{v(F)-2}),
\]
where the error term comes from \eqref{eq:Ps-derivative-bounds}. The value $P_s(z_0)$ is $\N(F,T_{s,2})$. Moreover,
\[
        P_s'(z_0)=
        \begin{cases}
        \dfrac{2k}{|\Aut(F)|}\left(\dfrac{s^2}{4}\right)^{k-1}+O_F(s^{v(F)-3}), & \nu=0,\\[1.4ex]
        \dfrac{k}{|\Aut(F)|}s\left(\dfrac{s^2}{4}\right)^{k-1}+O_F(s^{v(F)-3}), & \nu=1.
        \end{cases}
\]
Using the definition of $\kappa_F$, we obtain
\begin{equation}
        \N(F,K_{x,y})\le \N(F,T_{s,2})
        -\kappa_F\eta s^{v(F)-2}+O_{F,C_0}(s^{v(F)-2}).
        \label{eq:general-graph-loss}
\end{equation}

\medskip
\noindent\emph{\textbf{Step 2: the loss caused by missing cross-edges.}}
Let $D=E(K_{x,y})\setminus E(H)$, so $|D|=\mu$. For the assignment $A\to X$, $B\to Y$, let $\mathcal C_{X,Y}(J)$ be the number of injective maps $\phi:V(F)\to X\cup Y$ with $\phi(A)\subseteq X$ and $\phi(B)\subseteq Y$ such that every edge of the fixed matching $M$ is mapped to an edge of $J$, where $J\subseteq K_{x,y}$. Define $\mathcal C_{Y,X}(J)$ analogously.

The counts $\mathcal C_{X,Y}(H)$ and $\mathcal C_{Y,X}(H)$ may include maps which are not copies of $F$ in $H$, because only the matching edges are checked. This relaxation is deliberate: the matching-admissibility assumption ensures that a fixed matching saturates the smaller colour class, so the first-order loss caused by missing cross-edges is already detected by the images of these matching edges. Every labelled copy of $F$ in $H$ is counted in one of these two quantities. Hence
\begin{equation}\label{eq:C-upper-N}
        |\Aut(F)|\N(F,H)
        \le \mathcal C_{X,Y}(H)+\mathcal C_{Y,X}(H).
\end{equation}
For $J=K_{x,y}$, equality holds after summing the two assignments, because every injective map sending the two colour classes of $F$ to the two parts of $K_{x,y}$ is an embedding of $F$.

Fix the assignment $A\to X$, $B\to Y$. For a missing edge $e\in D$, let $\mathcal U_e$ be the set of maps counted by $\mathcal C_{X,Y}(K_{x,y})$ for which some edge of $M$ is mapped to $e$. Then
\[
        \mathcal C_{X,Y}(H)
        =\mathcal C_{X,Y}(K_{x,y})-
        \left|\bigcup_{e\in D}\mathcal U_e\right|.
\]
If $e=uv$ with $u\in X$ and $v\in Y$, then
\[
        |\mathcal U_e|=
        \begin{cases}
        k(x-1)_{k-1}(y-1)_{k-1}, & \nu=0,\\[0.8ex]
        k(x-1)_k(y-1)_{k-1}, & \nu=1.
        \end{cases}
\]
Indeed, one first chooses the matching edge of $M$ mapped to $e$. When $\nu=0$, all vertices of $F$ are covered by the matching, so the remaining $k-1$ vertices of $A$ and the remaining $k-1$ vertices of $B$ are placed injectively in $X\setminus\{u\}$ and $Y\setminus\{v\}$. When $\nu=1$, there is one additional unmatched vertex in $A$, so one places $k$ remaining vertices of $A$ and $k-1$ remaining vertices of $B$.

We now bound the overcount in the union. If $k=1$, no map can use two distinct missing edges as images of matching edges. Suppose $k\ge2$ and take two distinct missing edges $e,f\in D$. If they share an endpoint, then $\mathcal U_e\cap\mathcal U_f=\emptyset$, since the matching edges of $M$ are vertex-disjoint and the map is injective. If they are vertex-disjoint, then after fixing two matching edges and their images, at most $v(F)-4$ vertices remain to be placed. Hence
\[
        |\mathcal U_e\cap\mathcal U_f|=O_F(s^{v(F)-4}).
\]
Since $\mu=O_{C_0}(s)$, the sum of all pairwise intersections is $O_{F,C_0}(s^{v(F)-2})$. Inclusion-exclusion therefore gives
\begin{align*}
\mathcal C_{X,Y}(H)
&\le \mathcal C_{X,Y}(K_{x,y})
-k\mu(x-1)_{k-1}(y-1)_{k-1}
+O_{F,C_0}(s^{2k-2}), && \nu=0,\\
\mathcal C_{X,Y}(H)
&\le \mathcal C_{X,Y}(K_{x,y})
-k\mu(x-1)_k(y-1)_{k-1}
+O_{F,C_0}(s^{2k-1}), && \nu=1.
\end{align*}
The same estimates hold with $x$ and $y$ interchanged for the assignment $A\to Y$, $B\to X$.

It remains to identify the leading coefficient after the two assignments are added. For $k\ge2$,
\[
        (x-1)_{k-1}(y-1)_{k-1}
        =\prod_{i=1}^{k-1}(z-is+i^2)
        =\left(\frac{s^2}{4}\right)^{k-1}+O_{F,C_0}(s^{2k-3}),
\]
and for $k=1$ the same expression is the empty product $1=(s^2/4)^0$. Thus, when $\nu=0$, the total leading loss from the two assignments is
\[
        2k\mu\left(\frac{s^2}{4}\right)^{k-1}
        +O_{F,C_0}(s^{2k-2}).
\]
When $\nu=1$, we use
\[
\begin{aligned}
        &(x-1)_k(y-1)_{k-1}+(y-1)_k(x-1)_{k-1} \\
        &\qquad =\bigl((x-k)+(y-k)\bigr)(x-1)_{k-1}(y-1)_{k-1} \\
        &\qquad =(s-2k)(x-1)_{k-1}(y-1)_{k-1}.
\end{aligned}
\]
Together with \eqref{eq:xy-near-balanced}, this gives total leading loss
\[
        k\mu s\left(\frac{s^2}{4}\right)^{k-1}
        +O_{F,C_0}(s^{2k-1}).
\]
Combining these estimates with \eqref{eq:C-upper-N} and dividing by $|\Aut(F)|$, we get
\begin{equation}
        \N(F,H)\le \N(F,K_{x,y})
        -\kappa_F\mu s^{v(F)-2}+O_{F,C_0}(s^{v(F)-2}).
        \label{eq:general-missing-loss}
\end{equation}

\medskip
\noindent\emph{\textbf{Step 3: combining the two defects.}}
Combining \eqref{eq:general-missing-loss} with the imbalance estimate \eqref{eq:general-graph-loss} for $\N(F,K_{x,y})$ yields
\[
        \N(F,H)\le \N(F,T_{s,2})
        -\kappa_F(\eta+\mu)s^{v(F)-2}+O_{F,C_0}(s^{v(F)-2}).
\]
Since $\eta+\mu=t_{s,2}-e(H)$, this is exactly the desired estimate.
\end{proof}

\section{Counting bounds for \texorpdfstring{$C_{2\ell+1}$}{C2l+1}-free graphs}\label{sec:counting-bounds}

We record the counting bounds used later. The first extends Theorem~\ref{thm:local-counting} to $C_{2\ell+1}$-free graphs with $t_{n,2}-O(n)$ edges.

\begin{proposition}\label{prop:lifting}
Let $F$ be a fixed matching-admissible connected bipartite graph, with coefficient $\kappa_F$ as in \eqref{eq:kappa-F}. Fix $\ell\ge2$ and $A>0$. There is a constant $C=C(F,\ell,A)$ such that the following holds for all sufficiently large $n$. If $G$ is an $n$-vertex $C_{2\ell+1}$-free graph with $e(G)\ge t_{n,2}-An$, then
\[
        \N(F,G)\le \N(F,T_{n,2})-
        \kappa_F(t_{n,2}-e(G))n^{v(F)-2}+Cn^{v(F)-2}.
\]
\end{proposition}

\begin{proof}
Choose an integer $c=c(A)$ large enough that $e(G)\ge(n-c)^2/4$ for all sufficiently large $n$. By Lemma~\ref{lem:zlp-core}, there is a set $W\subseteq V(G)$ such that $H=G[W]$ is bipartite, $s=|W|=n-O_c(1)$, and each vertex outside $W$ has fewer than $11c$ neighbors in $W$. Since $|V(G)\setminus W|=O_c(1)$, the number of edges from $V(G)\setminus W$ to $W$ is $O_c(1)$, and the number of edges inside $V(G)\setminus W$ is also $O_c(1)$. Thus only $O_c(1)$ edges of $G$ are incident with $V(G)\setminus W$. Every fixed edge lies in at most $O_F(n^{v(F)-2})$ copies of $F$, so $\N(F,G)\le \N(F,H)+O_{F,c}(n^{v(F)-2})$.
Then $e(H)=e(G)-O_c(1)$ and $t_{s,2}-e(H)=O_{A,c}(s)$. Theorem~\ref{thm:local-counting} gives
\begin{equation}\label{eq:lifting-H-bound}
        \N(F,H)\le \N(F,T_{s,2})-
        \kappa_F(t_{s,2}-e(H))s^{v(F)-2}+O_{F,A,c}(s^{v(F)-2}).
\end{equation}
Since $s=n-O_c(1)$, we have
\begin{equation}\label{eq:lifting-defect-relation}
        t_{s,2}-e(H)=(t_{n,2}-e(G))-(t_{n,2}-t_{s,2})+O_c(1),
        \qquad t_{n,2}-t_{s,2}=O_c(n).
\end{equation}
Let $J$ be the spanning subgraph of $T_{n,2}$ obtained by keeping a copy of $T_{s,2}$ and deleting all edges incident with the remaining $n-s$ vertices. Then $e(J)=t_{s,2}$ and $\N(F,J)=\N(F,T_{s,2})$. Applying Theorem~\ref{thm:local-counting} to $J\subseteq T_{n,2}$ gives
\begin{equation}\label{eq:lifting-Ts-bound}
        \N(F,T_{s,2})\le \N(F,T_{n,2})-
        \kappa_F(t_{n,2}-t_{s,2})n^{v(F)-2}+O_{F,c}(n^{v(F)-2}).
\end{equation}
Since $s=n-O_c(1)$ and $t_{s,2}-e(H)=O_{A,c}(n)$, replacing $s^{v(F)-2}$ by $n^{v(F)-2}$ in \eqref{eq:lifting-H-bound} changes the right-hand side by at most $O_{F,A,c}(n^{v(F)-2})$. Combining \eqref{eq:lifting-H-bound} and \eqref{eq:lifting-Ts-bound}, and using \eqref{eq:lifting-defect-relation}, proves the proposition.
\end{proof}

\begin{corollary}\label{cor:global-defects}
Fix $\ell\ge2$ and $A>0$. For all sufficiently large $n$ the following hold.
\begin{enumerate}[label=\textup{(\roman*)}]
\item If $t\ge4$ and $G$ is an $n$-vertex $C_{2\ell+1}$-free graph with $e(G)\ge t_{n,2}-An$, then
\[
        \N(P_t,G)\le \N(P_t,T_{n,2})-
        \frac{\lfloor t/2\rfloor}{2^{t-2}}(t_{n,2}-e(G))n^{t-2}
        +O_{\ell,t,A}(n^{t-2}).
\]
\item If $a\ge2$ and $G$ is an $n$-vertex $C_{2\ell+1}$-free graph with $e(G)\ge t_{n,2}-An$, then
\[
        \N(C_{2a},G)\le \N(C_{2a},T_{n,2})-
        \frac{1}{2^{2a-1}}(t_{n,2}-e(G))n^{2a-2}+O_{\ell,a,A}(n^{2a-2}).
\]
\end{enumerate}
\end{corollary}

\begin{proof}
Apply Proposition~\ref{prop:lifting}. For the path bound, take $F=P_t$. Then $v(F)=t$, $|\Aut(F)|=2$, and \eqref{eq:kappa-F} gives $\kappa_F=\lfloor t/2\rfloor/2^{t-2}$. For the cycle bound, take $F=C_{2a}$; here $k=a$, $\nu=0$ and $|\Aut(F)|=4a$, so $\kappa_F=1/2^{2a-1}$.
\end{proof}

The preceding bounds apply only after one knows that the graph has $t_{n,2}-O(n)$ edges. The following lemma gives this initial information from a near-extremal number of path or even-cycle copies.

\begin{lemma}\label{lem:spectral-init}
Let $\ell\ge2$ and $K>0$ be fixed.
\begin{enumerate}[label=\textup{(\roman*)}]
\item For every $t\ge4$ there is a constant $A=A(\ell,t,K)$ such that every $n$-vertex $C_{2\ell+1}$-free graph $G$ satisfying
\[
        \N(P_t,G)\ge \frac{n^t}{2^t}-Kn^{t-1}
\]
has $e(G)\ge t_{n,2}-An$ for all sufficiently large $n$.
\item For every $a\ge2$ there is a constant $A=A(\ell,a,K)$ such that every $n$-vertex $C_{2\ell+1}$-free graph $G$ satisfying
\[
        \N(C_{2a},G)\ge \frac{n^{2a}}{2a\,2^{2a}}-Kn^{2a-1}
\]
has $e(G)\ge t_{n,2}-An$ for all sufficiently large $n$.
\end{enumerate}
\end{lemma}

\begin{proof}
Let $A_G$ be the adjacency matrix of $G$, let $\lambda=\lambda(G)$, and let $\mathbf1$ be the all-one vector. If the counted graph is $P_t$, then every copy gives two injective walks of length $t-1$, whence
\[
        2\N(P_t,G)\le \mathbf1^TA_G^{t-1}\mathbf1
        \le \|\mathbf1\|^2\|A_G\|^{t-1}=n\lambda^{t-1}.
\]
By Lemma~\ref{lem:zls}, $\lambda\le\sqrt{e(G)}+\ell$. Therefore
\[
        \frac{n^t}{2^t}-Kn^{t-1}
        \le \frac n2(\sqrt{e(G)}+\ell)^{t-1}.
\]
Taking $(t-1)$st roots gives $\sqrt{e(G)}\ge n/2-O_{\ell,t,K}(1)$, and hence $e(G)\ge t_{n,2}-A n$ for a suitable $A$.

Now suppose the counted graph is $C_{2a}$. Every copy of $C_{2a}$ gives $4a$ closed walks of length $2a$, so
\[
        4a\N(C_{2a},G)\le \operatorname{tr} A_G^{2a}.
\]
If $\lambda_1,\ldots,\lambda_n$ are the eigenvalues of $A_G$, then
\[
        \operatorname{tr} A_G^{2a}=\sum_i\lambda_i^{2a}
        \le \lambda^{2a-2}\sum_i\lambda_i^2=2e(G)\lambda^{2a-2}.
\]
Using Lemma~\ref{lem:zls} again, we obtain
\[
        \frac{n^{2a}}{2^{2a-1}}-O_{a,K}(n^{2a-1})
        \le 2e(G)(\sqrt{e(G)}+\ell)^{2a-2}.
\]
If $e(G)<t_{n,2}-An$ with $A$ sufficiently large in terms of $\ell,a,K$, then the right-hand side is at most
\[
        \frac{n^{2a}}{2^{2a-1}}-\Omega(A)n^{2a-1},
\]
a contradiction. This proves the cycle part.
\end{proof}

\section{Proofs for paths}\label{sec:proofs-paths}

\begin{proof}[Proof of Theorem~\ref{thm:path-main}]
Let $G$ be an $n$-vertex $C_{2\ell+1}$-free graph with $\N(P_t,G)\ge b^{\mathrm P}_{t,r}(n)$. Since $\N(P_t,T_{n,2})=n^t/2^t+O(n^{t-1})$, Lemma~\ref{lem:path-extremal-asymp} implies $\N(P_t,G)\ge n^t/2^t-O(n^{t-1})$. Lemma~\ref{lem:spectral-init}(i) gives $e(G)\ge t_{n,2}-An$ for a constant $A=A(\ell,t,r)$.

We claim that
\begin{equation}\label{eq:path-edge-threshold}
        e(G)\ge \left\lfloor\frac{(n-r)^2}{4}\right\rfloor+\binom{r+1}{2}.
\end{equation}
Suppose not. Since both sides of \eqref{eq:path-edge-threshold} are integers,
\[
        e(G)\le \left\lfloor\frac{(n-r)^2}{4}\right\rfloor+\binom{r+1}{2}-1.
\]
Thus
\begin{equation}\label{eq:path-edge-gap-lower}
\begin{aligned}
        t_{n,2}-e(G)
        &\ge \left\lfloor\frac{n^2}{4}\right\rfloor
        -\left\lfloor\frac{(n-r)^2}{4}\right\rfloor
        -\binom{r+1}{2}+1\\
        &=\frac r2 n-O_r(1).
\end{aligned}
\end{equation}
Applying Corollary~\ref{cor:global-defects}(i) and \eqref{eq:path-edge-gap-lower}, we obtain
\begin{equation}\label{eq:path-count-upper-defect}
        \N(P_t,G)\le \N(P_t,T_{n,2})-
        \frac{\lfloor t/2\rfloor r}{2^{t-1}}n^{t-1}+O(n^{t-2}).
\end{equation}
By Lemma~\ref{lem:path-extremal-asymp},
\[
        b^{\mathrm P}_{t,r}(n)=\N(P_t,T_{n,2})-\frac{t(r-1)}{2^t}n^{t-1}+O(n^{t-2}).
\]
If $t$ is even, then $\lfloor t/2\rfloor r/2^{t-1}=tr/2^t>t(r-1)/2^t$; if $t$ is odd and $r\le t-1$, then $\lfloor t/2\rfloor r/2^{t-1}=(t-1)r/2^t>t(r-1)/2^t$. Hence \eqref{eq:path-count-upper-defect} contradicts $\N(P_t,G)\ge b^{\mathrm P}_{t,r}(n)$ for all sufficiently large $n$. This proves \eqref{eq:path-edge-threshold}.

By Theorem~\ref{thm:zlp} and \eqref{eq:path-edge-threshold}, either $G\in\cG_{n,r+1}$, or $G$ is obtained from $T_{n-r,2}$ by suspending $K_{r+1}$. In the second case, the bipartite core has $n-r$ vertices, and the copies of $P_t$ using at least one of the $r$ outside vertices contribute only $O(n^{t-2})$. Thus $\N(P_t,G)=\N(P_t,T_{n-r,2})+O(n^{t-2})$. By \eqref{eq:even-T-diff} and \eqref{eq:odd-T-diff}, this gives
\[
        \N(P_t,G)=\N(P_t,T_{n,2})-\frac{tr}{2^t}n^{t-1}+O(n^{t-2}),
\]
which is smaller than $b^{\mathrm P}_{t,r}(n)$ for all sufficiently large $n$. Thus the second case cannot occur, and $G\in\cG_{n,r+1}$.

If $G\notin\cG_{n,r}$, then $G\in\cG_{n,r+1}\setminus\cG_{n,r}$, and the assumed lower bound gives $G\in\cB^{\mathrm P}_{t,r}(n)$. Proposition~\ref{prop:path-extremal} gives $G\in\cS_{t,r}(n)$.
\end{proof}

Thus Corollary~\ref{cor:path-chromatic} follows directly from Theorem~\ref{thm:path-main}, except when $t$ is odd and $r\ge t$. In this case a bound depending only on $t_{n,2}-e(G)$ does not give the required contradiction. We use the following lemma, which also accounts for the vertices outside the bipartite core given by Lemma~\ref{lem:zlp-core}.

\begin{lemma}\label{lem:odd-two-param}
Fix $\ell\ge2$, an odd integer $t\ge5$ and a constant $A>0$. There is a constant $C=C(\ell,t,A)$ such that the following holds for all sufficiently large $n$. Let $G$ be an $n$-vertex $C_{2\ell+1}$-free graph with $e(G)\ge t_{n,2}-An$. Choose an integer $c=c(A)\ge2$ with $c>2A$, and let $W$ be the set given by Lemma~\ref{lem:zlp-core} with this value of $c$. Let $H=G[W]$, $s=|W|=n-\tau$, and $\rho=t_{s,2}-e(H)$.
Then
\[
        \N(P_t,G)\le \N(P_t,T_{n,2})
        -\frac{t\tau}{2^t}n^{t-1}
        -\frac{(t-1)\rho}{2^{t-1}}n^{t-2}
        +Cn^{t-2}.
\]
\end{lemma}

\begin{proof}
Since $c>2A$, we have $t_{n,2}-An\ge (n-c)^2/4$ for all sufficiently large $n$, so Lemma~\ref{lem:zlp-core} applies with this value of $c$.
The set $W$ has size $n-O(1)$ and every vertex outside $W$ has $O(1)$ neighbors in $W$. Hence there are only $O(1)$ edges from $V(G)\setminus W$ to $W$, and only $O(1)$ edges inside $V(G)\setminus W$. Therefore the number of copies of $P_t$ using at least one edge not contained in $H$ is $O(n^{t-2})$. Thus $\N(P_t,G)\le \N(P_t,H)+O(n^{t-2})$.
Since $\rho=t_{s,2}-e(H)=O_A(s)$, Theorem~\ref{thm:local-counting} applies to $H$ with $F=P_t$. Since $t$ is odd, \eqref{eq:kappa-F} gives $\kappa_F=(t-1)/2^{t-1}$. Using $s=n-\tau=n-O(1)$, we get
\begin{equation}\label{eq:odd-two-param-H-count}
        \N(P_t,H)
        \le \N(P_t,T_{s,2})-
        \frac{(t-1)\rho}{2^{t-1}}n^{t-2}+O(n^{t-2}).
\end{equation}
Finally, \eqref{eq:odd-T-diff} gives
\begin{equation}\label{eq:odd-two-param-core-count}
        \N(P_t,T_{s,2})
        =\N(P_t,T_{n,2})-
        \frac{t\tau}{2^t}n^{t-1}+O(n^{t-2}).
\end{equation}
Combining $\N(P_t,G)\le \N(P_t,H)+O(n^{t-2})$ with \eqref{eq:odd-two-param-H-count} and \eqref{eq:odd-two-param-core-count} proves the lemma.
\end{proof}

\begin{proof}[Proof of Corollary~\ref{cor:path-chromatic}]
Let $S\in\cS_{t,r}(n)$. By Proposition~\ref{prop:path-extremal}, $\N(P_t,S)=b^{\mathrm P}_{t,r}(n)=b^{\mathrm P,\chi}_{t,r}(n)$.
Moreover, each member $S'\in\cS_{t,r}(n)$ is $C_{2\ell+1}$-free and satisfies $\chi(S')=r$. We show that every $n$-vertex $C_{2\ell+1}$-free graph $G$ with $\chi(G)\ge r$ and $\N(P_t,G)\ge\N(P_t,S)$ belongs to $\cS_{t,r}(n)$.

If $t$ is even, or if $t$ is odd and $r\le t-1$, then Theorem~\ref{thm:path-main} gives $G\in\cG_{n,r}$ or $G\in\cS_{t,r}(n)$. Since every graph in $\cG_{n,r}$ is $(r-1)$-colourable, so the condition $\chi(G)\ge r$ excludes the first case. Hence $G\in\cS_{t,r}(n)$.

Now assume that $t$ is odd and $r\ge t$. By Lemma~\ref{lem:path-extremal-asymp},
\[
        b^{\mathrm P,\chi}_{t,r}(n)
        =\N(P_t,T_{n,2})-\frac{t(r-1)}{2^t}n^{t-1}+O(n^{t-2}).
\]
The lower bound $\N(P_t,G)\ge b^{\mathrm P,\chi}_{t,r}(n)$ and Lemma~\ref{lem:spectral-init}(i) imply $e(G)\ge t_{n,2}-An$ for some constant $A=A(\ell,t,r)$.
Choose $c=c(A)$ as in Lemma~\ref{lem:odd-two-param}, let $W$ be the set given by Lemma~\ref{lem:zlp-core}, let $H=G[W]$, $|W|=n-\tau$, and set $\rho=t_{n-\tau,2}-e(H)$. Since $H$ is bipartite and the $\tau$ vertices outside $W$ may be given distinct additional colours, $\chi(G)\le \tau+2$. Hence
\begin{equation}\label{eq:tau-lower}
        \tau\ge r-2.
\end{equation}

We next prove the edge threshold \eqref{eq:path-edge-threshold}. Suppose it fails. Then $t_{n,2}-e(G)\ge (r/2)n-O_r(1)$. Since only $O(1)$ edges of $G$ are not contained in $H$, we have
\[
        t_{n,2}-e(G)
        =t_{n,2}-t_{n-\tau,2}+\rho+O(1)
        =\frac{\tau}{2}n+\rho+O(1).
\]
Since $H$ is bipartite, $\rho=t_{n-\tau,2}-e(H)\ge0$. If $\tau\ge r$, Lemma~\ref{lem:odd-two-param} gives
\[
        \N(P_t,G)\le \N(P_t,T_{n,2})
        -\frac{tr}{2^t}n^{t-1}+O(n^{t-2}),
\]
contradicting the assumed lower bound for all sufficiently large $n$. Therefore $\tau\le r-1$, and the preceding two inequalities yield
\begin{equation}\label{eq:rho-lower}
        \rho\ge\frac{r-\tau}{2}n-O_r(1).
\end{equation}

Now Lemma~\ref{lem:odd-two-param} and \eqref{eq:rho-lower} give
\[
\begin{aligned}
        \N(P_t,G)
        &\le \N(P_t,T_{n,2})
        -\left(\frac{t\tau}{2^t}+\frac{(t-1)(r-\tau)}{2^t}\right)n^{t-1}
        +O(n^{t-2})\\
        &=\N(P_t,T_{n,2})-\frac{(t-1)r+\tau}{2^t}n^{t-1}+O(n^{t-2}).
\end{aligned}
\]
By \eqref{eq:tau-lower}, the coefficient of $n^{t-1}$ in the loss is at least $(tr-2)/2^t$, which is larger than $t(r-1)/2^t$ since $t\ge5$. This again contradicts $\N(P_t,G)\ge b^{\mathrm P,\chi}_{t,r}(n)$ for all sufficiently large $n$. Hence \eqref{eq:path-edge-threshold} holds.

By Theorem~\ref{thm:zlp} and \eqref{eq:path-edge-threshold}, either $G\in\cG_{n,r+1}$, or $G$ is obtained from $T_{n-r,2}$ by suspending $K_{r+1}$. In the second case, $\N(P_t,G)=\N(P_t,T_{n-r,2})+O(n^{t-2})$, and \eqref{eq:even-T-diff} and \eqref{eq:odd-T-diff} give
\[
        \N(P_t,G)=\N(P_t,T_{n,2})-\frac{tr}{2^t}n^{t-1}+O(n^{t-2}),
\]
which is smaller than $b^{\mathrm P,\chi}_{t,r}(n)$ for all sufficiently large $n$. Hence $G\in\cG_{n,r+1}$.

Since every graph in $\cG_{n,r}$ is $(r-1)$-colourable, so $\chi(G)\ge r$ gives $G\notin\cG_{n,r}$. Thus $G\in\cG_{n,r+1}\setminus\cG_{n,r}$ and $\chi(G)\ge r$. By the definition of $b^{\mathrm P,\chi}_{t,r}(n)$ and the assumed lower bound, $G\in\cB^{\mathrm P,\chi}_{t,r}(n)$. Proposition~\ref{prop:path-extremal} therefore gives $G\in\cS_{t,r}(n)$.
\end{proof}

\section{Proofs for even cycles}\label{sec:proofs-cycles}

\begin{proof}[Proof of Theorem~\ref{thm:cycle-main}]
Let $G$ be an $n$-vertex $C_{2\ell+1}$-free graph with $\N(C_{2a},G)\ge b^{\circ}_{a,r}(n)$. Since $\N(C_{2a},T_{n,2})=n^{2a}/(2a\,2^{2a})+O(n^{2a-1})$, Lemma~\ref{lem:cycle-extremal-asymp} implies $\N(C_{2a},G)\ge n^{2a}/(2a\,2^{2a})-O(n^{2a-1})$. Lemma~\ref{lem:spectral-init}(ii) gives $e(G)\ge t_{n,2}-An$ for a constant $A=A(\ell,a,r)$.

We claim that
\begin{equation}\label{eq:cycle-edge-threshold}
        e(G)\ge \left\lfloor\frac{(n-r)^2}{4}\right\rfloor+\binom{r+1}{2}.
\end{equation}
Suppose not. The same integer calculation as in \eqref{eq:path-edge-gap-lower} gives $t_{n,2}-e(G)\ge (r/2)n-O_r(1)$.
Applying Corollary~\ref{cor:global-defects}(ii), we obtain
\[
        \N(C_{2a},G)\le \N(C_{2a},T_{n,2})-\frac{r}{2^{2a}}n^{2a-1}+O(n^{2a-2}).
\]
By Lemma~\ref{lem:cycle-extremal-asymp},
\[
        b^{\circ}_{a,r}(n)=\N(C_{2a},T_{n,2})-\frac{r-1}{2^{2a}}n^{2a-1}+O(n^{2a-2}),
\]
which contradicts $\N(C_{2a},G)\ge b^{\circ}_{a,r}(n)$ for all sufficiently large $n$. This proves \eqref{eq:cycle-edge-threshold}.

By Theorem~\ref{thm:zlp} and \eqref{eq:cycle-edge-threshold}, either $G\in\cG_{n,r+1}$, or $G$ is obtained from $T_{n-r,2}$ by suspending $K_{r+1}$. In the second case, every copy of $C_{2a}$ is either in the bipartite core or in the clique $K_{r+1}$, and hence
\[
        \N(C_{2a},G)=\N(C_{2a},T_{n-r,2})+O(1)
        =\N(C_{2a},T_{n,2})-\frac r{2^{2a}}n^{2a-1}+O(n^{2a-2}),
\]
which is smaller than $b^{\circ}_{a,r}(n)$. Thus the exceptional case cannot occur, and $G\in\cG_{n,r+1}$.

If $G\notin\cG_{n,r}$, then $G\in\cG_{n,r+1}\setminus\cG_{n,r}$. By the definition of $b^{\circ}_{a,r}(n)$, we have $\N(C_{2a},G)\le b^{\circ}_{a,r}(n)$; together with the assumed reverse inequality this gives $G\in\cB^{\circ}_{a,r}(n)$. Proposition~\ref{prop:cycle-extremal} identifies this second alternative explicitly: $G\in\cR_r(n)$ if $r<2a$, and $G\in\cT^*(r,n)$ if $r\ge2a$.
\end{proof}

\begin{proof}[Proof of Corollary~\ref{cor:cycle-chromatic}]
Fix $T\in \cT^*(r,n)$. By Proposition~\ref{prop:cycle-extremal}, $\N(C_{2a},T)=b^{\circ}_{a,r}(n)=b^{\circ,\chi}_{a,r}(n)$, and every $T'\in\cT^*(r,n)$ is $C_{2\ell+1}$-free and satisfies $\chi(T')=r$.

Let $G$ be an $n$-vertex $C_{2\ell+1}$-free graph with $\chi(G)\ge r$ and $\N(C_{2a},G)\ge \N(C_{2a},T)$. Theorem~\ref{thm:cycle-main} gives $G\in\cG_{n,r}$ or $G\in\cB^{\circ}_{a,r}(n)$. Since every graph in $\cG_{n,r}$ is $(r-1)$-colourable, the first case cannot occur. Hence $G\in\cB^{\circ}_{a,r}(n)$. Together with $\chi(G)\ge r$ and $b^{\circ}_{a,r}(n)=b^{\circ,\chi}_{a,r}(n)$, this gives $G\in\cB^{\circ,\chi}_{a,r}(n)$. Proposition~\ref{prop:cycle-extremal} then gives $G\in \cT^*(r,n)$.
\end{proof}

\section{Concluding remarks}\label{sec:odd-remark}

The restriction $r\le t-1$ in Theorem~\ref{thm:path-main} is needed only in the case where $t$ is odd. It comes from the following comparison. In the proof, if the edge bound \eqref{eq:path-edge-threshold} fails, then Corollary~\ref{cor:global-defects}(i) gives
\[
        \N(P_t,G)\le \N(P_t,T_{n,2})
        -\frac{\lfloor t/2\rfloor r}{2^{t-1}}n^{t-1}
        +O(n^{t-2}).
\]
On the other hand, Lemma~\ref{lem:path-extremal-asymp} gives
\[
        b^{\mathrm P}_{t,r}(n)=\N(P_t,T_{n,2})
        -\frac{t(r-1)}{2^t}n^{t-1}
        +O(n^{t-2}).
\]
Thus this argument contradicts $\N(P_t,G)\ge b^{\mathrm P}_{t,r}(n)$ only when $\lfloor t/2\rfloor r/2^{t-1}>t(r-1)/2^t$. For odd $t$, this is equivalent to $r<t$, which explains the restriction $r\le t-1$.

\begin{problem}\label{prob:odd-path-no-chromatic}
Fix $\ell\ge2$, and let $t\ge5$ be odd and $t\le r\le2\ell-1$. Is it true that, for all sufficiently large $n$, every $n$-vertex $C_{2\ell+1}$-free graph $G$ with $\N(P_t,G)\ge b^{\mathrm P}_{t,r}(n)$ satisfies $G\in\cG_{n,r}$ or $G\in\cS_{t,r}(n)$?
\end{problem}

\end{document}